\documentclass[a4paper,12pt,intlimits,oneside]{amsart}

\usepackage{enumerate}
\usepackage{latexsym,amssymb}

\newcommand{\comment}[1]{}
\numberwithin{equation}{section}

\newcommand{\bR}{{\mathbb R}}

\newcommand{\R}{{\mathbb R}}

\def\A{{\mathcal A}}

\def\H{{\mathcal H}}

\def\N{{\mathfrak N}}

\def\a{{\mathfrak a}}

\def\A{\texttt{A}}

\def\BMO{B\! M\! O}

\def\dst{\displaystyle}

\newtheorem{thm}{Theorem}[section]
\newtheorem{prop}[thm]{Proposition}
\newtheorem{cor}[thm]{Corollary}
\newtheorem{lem}[thm]{Lemma}
\newtheorem{defn}[thm]{Definition}
\newtheorem{remark}[thm]{Remark}

\def\div{{\mbox{\small\rm  div}}\,}
\def\curl{{\mbox{\small\rm  curl}}\,}

\begin{document}

\title[]{Atomic decomposition and weak factorization in 
generalized Hardy spaces of closed forms}

\author[A. Bonami]{Aline Bonami}
\address{F\'ed\'eration Denis Poisson, MAPMO CNRS-UMR 7349,
 Universit\'e d'Orl\'eans,
45067 Orl\'eans Cedex 2, France}
\email{{\tt aline.bonami@gmail.com}}
\author[J. Feuto]{Justin Feuto}
\address{Laboratoire de Math\'ematiques Fondamentales, UFR Math\'ematiques et Informatique, Universit\'e F\'elix Houphou\"et-Boigny (Abidjan), 22 B.P 1194 Abidjan 22. C\^ote d'Ivoire}
\email{{\tt justfeuto@yahoo.fr}}
\author[S. Grellier]{Sandrine Grellier}
\address{F\'ed\'eration Denis Poisson, MAPMO CNRS-UMR 7349,
 Universit\'e d'Orl\'eans,
45067 Orl\'eans Cedex 2, France} 
\email{{\tt Sandrine.Grellier@univ-orleans.fr}}
\author[L.D. Ky]{Luong Dang Ky}
\address {Department of Mathematics, University of Quy Nhon, 170 An Duong Vuong, Quy Nhon, Binh Dinh, Viet Nam}
\email{luongdangky@qnu.edu.vn}

\thanks{}
\begin{abstract}
We give an atomic decomposition of closed forms on $\bR^n$, the coefficients of which belong to some Hardy space of Musielak-Orlicz type. These spaces are natural generalizations of weighted  Hardy-Orlicz spaces, when the Orlicz function depends on the space variable. One of them, called  $\mathcal H^{\rm log}$, appears naturally when considering products of functions in the Hardy space $\mathcal H^1$ and in $BMO$. As a main consequence of the atomic decomposition, we obtain a weak factorization of closed forms whose coefficients are in  $\mathcal H^{\rm log}$. Namely, a closed form in  $\mathcal H^{\rm log}$ is the infinite sum of the wedge product between an exact form in the Hardy space $\mathcal H^1$ and an exact form in $BMO$. The converse result, which generalizes the classical div-curl lemma, is a consequence of \cite{BGK}. As a corollary, we prove that the real-valued $\H^{\log}$ space can be weakly factorized.
\end{abstract}
\maketitle

\section{Introduction}
Let $\varphi$ be a $\mathcal C^\infty$ function with compact support on $\R^n$ such that $\int\varphi dx=1$ and let $\varphi_t$ denote the dilated function $\varphi_t(x)=t^{-n}\varphi(x/t)$. The Hardy space $\H^p(\R^n)$ is defined as the space of distributions $f$ such that the function
\begin{equation}
\label{f+}f^+:=\sup_{t>0} | f*\varphi_t|
\end{equation}
is in $L^p(\R^n)$. It is well-known, from the seminal work of Fefferman and Stein \cite{FeS}, not only that the definition does not depend on the particular function $\varphi$, but that  the Hardy space $\H^p(\R^n)$ can be characterized in terms of the area or the grand maximal function. Its characterization in terms   of the atomic decomposition  for $0<p\leq 1$ as initiated by Coifman in \cite{C} when $n=1$ and Latter in \cite{L}  when $n>1$, revealed also a fundamental tool in the theory of Hardy spaces. Hardy spaces have been recently generalized in the context of Musielak-Orlicz spaces, first by the fourth author who proved the atomic decomposition \cite{Ky}, then by Dachun Yang et al. in \cite{HYY, LHY}, where other equivalent properties are proved. For the definition of these new spaces,   the Orlicz function $t^p$ is replaced by a function $\wp(x, t)$, which belongs, as a function of $x$ and uniformly in $t$, to the class of weights $A_\infty$, while, as a function of $t$ and uniformly in $x$, it belongs to the class of growth functions that had been introduced by Janson in \cite{J} to define Hardy-Orlicz  spaces. 

The Musielak-Orlicz-type space $L^{\wp}(\R^n)$ is the space of all Lebesgue measurable functions $f$ such that $$\int_{\mathbb R^{n}}\wp\left(x,\frac{\left|f(x)\right|}{\lambda}\right)dx<\infty$$ for some $\lambda>0$, with the Luxembourg (quasi-)norm
\begin{equation}\label{normeMusielackOrlicz}
\displaystyle \left\|f\right\|_{L^{\wp}}:=\inf\left\{\lambda>0:\int_{\mathbb R^{n}}\wp\left(x,\frac{\left|f(x)\right|}{\lambda}\right)dx\leq1\right\}.
\end{equation}
The Hardy space of Musielak-Orlicz type $\H^{\wp}(\R^n)$ consists of tempered distributions $f$ for which $f^+$ belongs to $L^{\wp}(\R^n)$. A typical example of such a function $\wp$ is given by
\begin{equation}\label{orlicz}
\theta(x,t):=\frac t{\log(e+|x|)+\log(e+t)}.
\end{equation} 
The corresponding Hardy space of Musielak-Orlicz type is denoted by $\mathcal H^{\rm log}(\R^n)$. This space appears naturally when considering products of  two functions, one in $\mathcal H^1(\R^n)$ and one in $BMO(\R^n)$. Namely, the product may be written as the sum of an integrable function and a distribution in   $\mathcal H^{\rm log}(\R^n)$ (see  \cite{BGK}). We will see that there is some kind of converse: all functions in $L^1(\R^n)+\H^{\log}(\R^n)$ may be written as the infinite sum of products of two functions, one in $\mathcal H^1(\R^n)$ and one in $BMO(\R^n)$. It is simple to write $L^1$ functions as the sum of such products. Indeed, we can restrict to functions $f$ such that $f=\sum \lambda_j \chi_{Q_j}$  where $Q_j$  are disjoint cubes. The characteristic function $\chi_{Q_j}$ can be written as $u_j^2$, so that $u_j$ has zero mean. Then $|Q_j|^{-1}u_j$ is an atom of $\H^1(\R^n)$ and
$$\sum_j \|u_j\|_{\H^1(\R^n)} \|v_j\|_{L^{\infty}(\R^n)}\leq C \|f\|_{L^1(\R^n)}.$$
We will concentrate here on the $\H^{\log}$ part.
\smallskip

 Before going back to real valued functions we study functions whose values are closed forms. We first define the   Hardy spaces of Musielak-Orlicz type for functions taking their values in the space of differential forms of a given degree. It is straightforward that the expression of  products of factors repectively in  $\mathcal H^1(\R^n)$ and  $BMO(\R^n)$ as a sum    extends to wedge products. But one can say more when one restricts to closed forms: in this case the cancellation properties imply that there is no integrable term. More precisely, the wedge product of two closed forms, respectively in $\H^{1}_{d}(\mathbb R^{n},\Lambda^{\ell})$ and $\BMO_d(\bR^n, \Lambda^m)$,  belongs to $\H^{\log}_{d}(\mathbb R^{n}, \Lambda^{\ell+m})$, with $1\leq \ell, m\leq d-1$. The index $d$ means that we restrict to closed forms. This may be seen as an $\H^1-BMO$ version of the generalized div-curl lemmas for wedge products of closed forms that have been obtained by Lou and McIntosh (see \cite{LM1}): they prove that, for $L^2$ data the wedge product has its coefficients in the Hardy space $\H^1(\R^n)$. A weaker $\H^1-BMO$ statement had been obtained  in \cite{BFG}.
\smallskip

The proof of this $\H^1-BMO$ generalized div-curl lemma for closed forms follows closely  what  was done in \cite{BGK} in the particular case of   the scalar product of divergence and curl free vector fields with coefficients respectively in $\H^1(\bR^n)$ and $BMO(\bR^n)$.  The main motivation of this paper concerns the converse, which is new even in the case of scalar products of divergence and curl free vector fields.
 More precisely, a closed form in  $\mathcal H^{\rm log}$ is the infinite sum of  wedge products between a closed form in the Hardy space $\mathcal H^1$ and a closed form in $BMO$, with a control on the norms. Restricting to $n$ forms which identify to real-valued functions, we obtain the  weak factorization that we mentioned earlier. 
 \begin{thm}\label{main3}
Let $n>1$. Any $f\in\H^{\log}(\bR^{n})$ can be written as
\begin{equation*}
f=\sum^{\infty}_{k=0}u_{k} v_{k}\qquad \qquad \text{ in the sense of distributions}
\end{equation*}
with $u_{k}\in \H^1(\bR^{n})$ and $v_{k}\in BMO(\bR^{n})$ and
\begin{equation}
\sum_k \|u_{k}\|_{\H^1(\bR^{n})}\| v_{k}\|_{BMO^+(\bR^{n})}\lesssim \|f\|_{\H^{\log}(\bR^{n})}.
\end{equation}
\end{thm}
Indeed, we prove that any atom can be written as a sum of $n$ scalar products $F\cdot G$ with $F$ which is $\div-$free with coefficients in $\H^{\log}(\R^n)$ and $G$ which is $\curl-$free with coefficients in $BMO(\R^n)$. This implies that each atom can be written as a sum of $n^2$ products $uv$, with $u\in \H^1(\bR^{n})$, $v\in BMO(\bR^{n})$ and $\|u\|_{\H^1(\bR^{n})}\| v\|_{BMO^+(\bR^{n})}$ uniformly bounded.

\begin{remark} For $n=1$ one has a better result. Indeed, we know that holomorphic functions of the space $\H^{\log}(\R^2_+)$ may be written as the product of holomorphic functions in $\H^1(\R^2_+)$ and  $BMOA(\R^2_+)$, see \cite{BK}. So each real valued function of $\H^{\log}(\R)$ is the sum of two products of functions, which are respectively in $\H^1 (\bR)$ and $BMO(\bR)$. 
\end{remark}

Theorem \ref{main3} gives a ``weak" answer to a question of \cite{BIJZ} on the surjectivity of the product as a map from $\H^1(\bR^{n})\times BMO(\bR^{n})$ to $L^1(\bR^{n})+\H^{\log}(\bR^n)$ (at that time the question concerned a larger space, for which one could not get a positive answer).
 
 \smallskip
 
 Again, the $L^2$ version of weak factorization for forms has been proved by Lou and McIntosh in \cite{LM1} and we follow the structure  of their proof. We first establish an atomic decomposition of distributions of all Hardy spaces of Musielak-Orlicz type $\H^{\wp}_{d}(\mathbb R^{n}, \Lambda^{\ell})$. Since the construction of \cite{Ky} does not preserve closed forms, we use the scheme of \cite{CMS}, already used by \cite{AMR} in the context of closed forms and Hardy spaces $\H^1$. Because of the possibility to use an atomic  decomposition, our main result deals with the way to write an atom of the space $\H_d^{\rm log}$  as a wedge product. The main step here is to find the $BMO$ factor corresponding to the size and position of the support of each atom.
 \smallskip

It should be emphazised that, as in \cite{BIJZ, BGK}, one of the main difficulties here is the fact that we deal with functions in $BMO$ and not equivalent classes modulo constants, since we are interested with ordinary products or wedge products. So, the boundedness inequalities do not only involve the $BMO$ semi-norm, but the norm
\begin{equation}\label{bmo+}
 \|f\|_{BMO^+}:=\int_{(0, 1)^n} |f(x)|dx +\sup_B\frac{1}{|B|}\int_B|f(x)-f_B|dx.
 \end{equation}
 Here the supremum is taken on all balls $B$ in $\R^n$ and $f_B$ denotes the mean of $f$ on $B$. The $BMO$ norm of (the equivalence class of) $f$ is the second term.
 \bigskip

The paper is organized as follows. Section 2 is devoted to atomic decompositions. In Section 3 we prove the generalized div-curl lemma and its converse, that is, weak factorization for the space $\H^{\log}$ of closed forms. 
\smallskip

Throughout this paper, we will denote by $C$, constants that depend only on the dimension $n$ and the growth function $\wp$. We use the standard notation $A\lesssim B$ to denote $A\leq CB$ for some $C$. And $A\sim B$ means that $A\lesssim B$ and $B\lesssim A.$ 
 If $E$ is a measurable subset of $\bR^{n}$, then $\left|E\right|$ stands for its Lebesgue measure.

\section{Prerequisites and atomic decomposition}

Prerequisites essentially rely on  \cite{Ky}, and also \cite {LHY,HYY}. We first give the required properties for the Orlicz-Musielak functions.

\subsection{Prerequisites on growth functions}
Let  $\wp:\mathbb R^{n}\times\left[0,\infty\right)\rightarrow\left[0,\infty\right)$ be a  Lebesgue measurable function. 
We say that $\wp$ 
   is of uniformly upper type $m$ (resp., lower type $m$) if there exists a positive constant $C$ such that 
\begin{equation*}
\wp(x,st)\leq C t^{m}\wp(x,s)%\label{control}
\end{equation*}
for every $t\geq 1$ (resp., $0<t\leq 1$).
Moreover we say that $\wp$ is of finite uniformly upper type (resp., positive lower type), if it is of  uniformly upper type for some $m<\infty$ (resp. lower type for some $m>0$).

For such a function $\wp$, we define the quantity 
$$\begin{array}{lll}
i(\wp)&:=&\sup\left\{m\in\mathbb R:\wp \text{ is of uniformly lower type }m\right\}.
\end{array}
$$

Let  $1< q<\infty$. We say that $\wp$ satisfies the uniformly Muckenhoupt condition $\mathbb A_{q}$ if there exists $C>0$ such that
 \begin{equation*}
 \frac{1}{\left|B\right|}\int_{B}\wp(x,t)dx\left(\frac{1}{\left|B\right|}\int_{B}\wp(x,t)^{-1/(q-1)}dx\right)^{q-1}\leq C
 \end{equation*}
for all balls $B\subset \R^n$ and $t\in [0,\infty)$, that is,  the functions $\wp(\cdot, t)$ belong uniformly to  the classical Muckenhoupt class $A_q$. 
% \begin{equation*}
% \frac{1}{\left|B\right|}\int_{B}\wp(x,t)dx\leq C\mathrm{ess}-\inf_{x\in B}\wp(x,t),\text{ if } q=1,
% \end{equation*}
% for all $t>0$ and all balls $B\subset\mathbb R^{n}$,
 We say that  $\wp$  is in $\mathbb A_{\infty}$ if there exists $1< q<\infty$ such that $\wp\in\mathbb A_{q}$. In this case, we also put 
 $$q(\wp):=\inf\left\{q> 1:\wp\in\mathbb A_{q}\right\}.$$
% \end{itemize}
\begin{defn}\label{growthFunction}
A measurable function $\wp: \bR^n\times\left[0,\infty\right)\rightarrow\left[0,\infty\right)$ is a growth function if it satisfies the three conditions
 \begin{enumerate}
 \item $\wp$ is a Musielak-Orlicz function, that is,
 the function $$\wp(x,\cdot):\left[0,\infty\right)\rightarrow\left[0,\infty\right)$$ is an Orlicz function for all $x\in\mathbb R^{n}$
 (i.e. $\wp(x,\cdot)$ is non decreasing, $\wp(x,0)=0$, $\wp(x,t)>0$ for $t>0$ and  $\lim_{t\to \infty}\wp(x,t)=\infty$),
\item  $\wp$ belongs to $\mathbb A_{\infty}$,
 \item  $\wp$ is of  uniformly lower type $p$ for some $p\in(0,1]$ and of uniformly upper type $1$.
 \end{enumerate}
 \end{defn}
 \begin{defn}\label{admissible}
For $\wp$ a growth function, we define $N_{\wp}=\left\lceil n(\frac{q(\wp)}{i(\wp)}-1)\right\rceil.$
A triple $(\wp,q,s)$ is called  admissible if $\wp$ is a growth function, $q>q(\wp)$ and $s$ is an integer greater or equal to $N_{\wp}$. 
\end{defn}
\begin{remark} The functions $\wp(x,t)=w(x)t^p$ are growth functions when $p\leq 1$ and $w$ is an $A_{\infty}$ weight. We are interested in generalizations of the corresponding weighted Hardy spaces.  
\end{remark}
\begin{remark} We restrict ourselves to upper type $1$ since our aim is to generalize $\H^p$ spaces for $p\le 1$. Lower type $1$ would allow to generalize those $\H^p$ spaces for which $p\ge 1$.
\end{remark}

Finally, for $1\leq q\leq\infty$ and for  $B$ a ball in $\mathbb R^{n}$, we denote by $L^{q}_{\wp}(B)$ the space of all measurable functions $f$ on $\mathbb R^{n}$ supported in $B$ such that $\left\|f\right\|_{L^{q}_{\wp}(B)}<\infty$, where
\begin{equation}\label{NormeLqP}
\left\|f\right\|_{L^{q}_{\wp}(B)}:=\left\{\begin{array}{lll}\displaystyle \sup_{t>0}\left(\frac{1}{\wp(B,t)}\int_{\mathbb R^{n}}\left|f(x)\right|^{q}\wp(x,t)dx\right)^{\frac{1}{q}}&\text{ if }& 1\leq q< \infty\\
\displaystyle\left\|f\right\|_{L^{\infty}}&\text{ if }&q=\infty
\end{array}
\right.
\end{equation}
with $\displaystyle \wp(B,t):=\int_{B}\wp(x,t)dx$.

We need the following properties, which are given in \cite{Ky}.
Let $\wp$ be a growth function. Then, for $(f_j)_{j=1}^\infty$,
\begin{equation}\label{concave}
  \int_{\R^n}\wp( x, \sum_j|f_j(x)|)dx\lesssim \sum_j  \int_{\R^n}\wp( x, |f_j(x)|)dx.
  \end{equation}
  For $q\geq 1$,  
  \begin{equation}\label{holder}
   \int_{B}\wp( x, |f(x)|)dx\lesssim \wp(B, \left\|f\right\|_{L^{q}_{\wp}(B)}).
  \end{equation}
\smallskip

\begin{remark} In \cite{Ky} an extra assumption of uniform local integrability is used to obtain Lemma 6.1.  However, one can verify that it is an easy consequence of the fact that $\wp\in\mathbb A_{q}$ for some $q>1$.
\end{remark}

\subsection{Real valued Hardy spaces of Musielak-Orlicz type}

Let us recall the definition of this new class of functions.
\begin{defn}
Let $\wp$ be a growth function. The Hardy space of Musielak-Orlicz type  $\mathcal H^{\wp}(\bR^{n})$ is the space of all tempered distributions $f$ such that $f^+$ is in $L^{\wp}(\mathbb R^{n})$ equipped with the quasi-norm 
$$\left\|f\right\|_{\mathcal H^{\wp}}:=\left\|f^+\right\|_{L^{\wp}}.$$
\end{defn}
Here $f^+$ is defined by \eqref{f+}. The fact that  the Hardy space $\H^{\wp}(\bR^{n})$ does not depend on the function $\varphi$ used in the definition of $f^+$ is proved in \cite{LHY}. It is also proved there that the grand maximal function of a distribution $f\in \H^{\wp}(\bR^{n})$ is in $L^{\wp}(\mathbb R^{n})$.
\smallskip
 
 Let us now describe  the atomic decomposition of a distribution in $\H^{\wp}(\bR^{n})$.
  We first recall the definition of an atom, as given in \cite{Ky}.
 
 \begin{defn}\label{atom-def}
 Let  $(\wp,q,s)$ be an admissible triple.   A measurable function $ \a$ is called a $(\wp,q,s)$-atom related to the ball $B$ if it satisfies the following conditions. 
\begin{itemize}
\item [(i)]  $ \a\in L^{q}_{\wp}(B)$,
\item [(ii)] $\displaystyle \left\| \a\right\|_{L^{q}_{\wp}(B)}\leq\left\|\chi_{B}\right\|^{-1}_{L^{\wp}}$, 
\item [(iii)] $\displaystyle \int_{\mathbb R^{n}} \a(x)x^{\alpha}dx=0 \text{ for every }\left|\alpha\right|\leq s$.
 \end{itemize}
\end{defn}
Here $\chi_B$ denotes the characteristic function of $B$.  Given a sequence $\left\{\a_{j}\right\}_{j}$ of multiples of $(\wp,q,s)$-atoms with $\a_{j}$ related to the ball $B_{j}$, we put
 \begin{equation}\label{N_q}
 \mathfrak N_{q}(\left\{\a_{j}\right\})=\inf\left\{\lambda>0:\sum_{j}\wp\left(B_{j},\frac{\left\|\a_{j}\right\|_{L^{q}_{\wp}(B_{j})}}{\lambda}\right )\leq 1\right\}.
 \end{equation}

 The atomic decomposition theorem (Theorem 3.1 in \cite{Ky}) can be stated as follows.  
 \begin{prop}\label{atomicDecomp} Let $(\wp,q,s)$ be an admissible triple.  Assume that  $\left\{\a_{j}\right\}^{\infty}_{j=1}$ is a sequence of multiples of $(\wp,q,s)$-atoms,   with each $\a_{j}$ related to the  ball $B_{j}$. Assume moreover  that
$$ \sum_{j}\wp\left(B_{j},\left\|\a_{j}\right\|_{L^{q}_{\wp}(B_{j})}\right )< \infty.$$
Then  the series $\sum_j \a_j$ converges in the distribution sense and in $\H^{\wp}(\bR^{n})$. Moreover,
$$\|\sum_j \a_j\|_{\H^{\wp}}\lesssim \mathfrak N_{q}(\left\{\a_{j}\right\})<\infty.$$
Conversely, for $f\in \H^{\wp}(\bR^{n}),$ there exists a sequence $\left\{\a_{j}\right\}^{\infty}_{j=1}$ of multiples of $(\wp,q,s)$-atoms, each $\a_{j}$ related to the  ball $B_{j}$, such that 
 \begin{equation}
 f=\sum_{j} \a_{j} \qquad \text{ in  }\, \mathcal S'(\mathbb R^{n}),\label{decompo1}
 \end{equation}
 with 
 \begin{equation*}
 \sum_{j}\wp(B_{j},\left\|\a_{j}\right\|_{L^{q}_{\wp}(B_{j})})<\infty.
 \end{equation*}
 Moreover, if $\left\|f\right\|_{\mathcal H^{\wp}(\bR^{n})}=1$, then
 \begin{equation*}
 \mathfrak N_{q}(\left\{\a_{j}\right\})\sim 1.
 \end{equation*}
 \end{prop}

\medskip

In this paper, we consider the corresponding spaces of closed $\ell$ forms, for  $1\leq\ell\leq n$.
\subsection{Hardy spaces of Musielak-Orlicz type of closed forms}

Let us  fix notations.
 
Let $\left\{e_{1},\ldots,e_{n}\right\}$ be the canonical orthonormal basis of the Euclidean space $\bR^{n}$ and $\left\{e^{1},\ldots,e^{n}\right\}$ be the corresponding dual basis. We  also use the classical notation $dx_1, \cdots, dx_n$ for the dual basis. For  $\ell\in\left\{1,\ldots,n\right\}$ we denote by $\Lambda^{\ell}$ the space  of  $\ell$-linear alternating forms. Its  canonical orthonormal basis is given by
\begin{equation*}
e^{I}=e^{i_{1}}\wedge\ldots\wedge e^{i_{\ell}},
\end{equation*}
 where $I$ varies among ordered $\ell$-tuples, $1\leq i_{1}<\ldots<i_{\ell}\leq n$. Let $\mathcal I_\ell$ denote the set of such ordered $\ell$-tuples.
  
 For every space $\mathcal E(\Omega)$ of functions or distributions, we easily define the space $\mathcal E(\Omega, \Lambda^\ell)$. Its elements may be written as 
 $$f=\sum_{I\in \mathcal I_\ell} f_I e^I,$$
 with $f_I$ in $\mathcal E(\Omega)$.  
 
 \smallskip
 
  In particular,   we define  the space $\mathcal H^{\wp}\left(\bR^{n},\Lambda^{\ell}\right)$ as the space of tempered distributions $f$ with values in $\Lambda^\ell$ such that $f^+$ is in $L^{\wp}(\R^n)$. In the definition of $f^+$, the absolute value is replaced by the (euclidean) norm in $\Lambda^\ell.$
 \smallskip
 
 All properties of Hardy spaces extend to these spaces. We will restrict to closed forms, that is, distributions $f$ such that $df=0$.
 Recall that the exterior derivative maps $\mathcal S'(\bR^{n},\Lambda^{\ell})$ into $\mathcal S'(\bR^{n},\Lambda^{\ell+1})$ and is defined, for
$f=\sum_{I\in\mathcal I_\ell}f_{I}e^{I}$, by
 \begin{equation*}
df:=\sum_{I\in\mathcal I_\ell}\sum_{j=1}^n\partial_{j}f_{I}e^j\wedge e^{I}.
 \end{equation*}
 The action of $f\in\mathcal S'(\bR^{n},\Lambda^{\ell})$ on $g\in \mathcal S(\bR^{n},\Lambda^{\ell})$ is given (with obvious notations) by
 $$\langle f, g\rangle:=\sum_{I\in\mathcal I_\ell} \langle f_I, g_I\rangle.$$
We also need the formal adjoint of the exterior derivative, which maps 
$\mathcal S'(\bR^{n},\Lambda^{\ell})$ into $\mathcal S'(\bR^{n},\Lambda^{\ell-1})$
and is defined by the identity
$$\langle\delta f, g\rangle =\langle f, dg\rangle.$$
Recall that the Hodge Laplacian is defined on $\mathcal S'(\bR^{n},\Lambda^{\ell})$ by
\begin{equation}\label{Hodge}
{\bf \Delta}=d\delta+\delta d.
\end{equation}
It coincides  with the ordinary Laplacian coefficient by coefficient.

	\bigskip

We can now define the  Hardy spaces  of Musielak-Orlicz type under consideration. 

\begin{defn}
 Let $\wp$ be a growth function. The Hardy space of Musielak-Orlicz type of closed $\ell$ forms, which we denote by   $\mathcal H^{\wp}_d \left(\bR^{n},\Lambda^{\ell}\right)$, is the space of all tempered distributions $f$ with values in $\Lambda^{\ell}$ such that $f$ is in $\H^{\wp}\left(\bR^{n},\Lambda^{\ell}\right)$ and $ df=0$.
\end{defn}

This generalizes the spaces $\H^{1}_{d}(\bR^{n},\Lambda^{\ell})$, which have been introduced in \cite {LM1}. Remark that this space reduces to $\{0\}$ when $\ell=0$ and identifies to  $\H^{\wp}(\bR^{n})$ when $\ell=n$. 

Since the convergence in $\mathcal H^{\wp} (\bR^{n})$ implies the convergence in the distribution sense, $\mathcal H^{\wp}_d \left(\bR^{n},\Lambda^{\ell}\right)$ is a closed subspace of $\mathcal H^{\wp} \left(\bR^{n},\Lambda^{\ell}\right)$.
\smallskip

The definition of atoms of $\mathcal H^{\wp}\left(\bR^{n},\Lambda^{\ell}\right)$  is the same as the one of real-valued atoms, except that now they take values in $\Lambda^{\ell}$. More precisely, a $(\wp,q,s, \ell)$-atom satisfies all properties given in Definition \ref{atom-def}, with $L^{q}_{\wp}(B)$ replaced by $L^{q}_{\wp}(B, \Lambda^{\ell})$. We now define atoms of $\mathcal H^{\wp}_d \left(\bR^{n},\Lambda^{\ell}\right)$. The only novelty is the requirement that they are closed.
\begin{defn} Let $1\leq \ell\leq n$. Let $\wp$ be a growth function and $(\wp,q,s)$ be an admissible triple.
A measurable function $\mathfrak a$ with values in $\Lambda^{\ell}$ is called a $(\wp,q,s, \ell)_{d}$-atom if it is a $(\wp,q,s, \ell)$-atom which satisfies $d\mathfrak a=0$. 
\end{defn}
\begin{remark} For $\ell=n$, a $(\wp,q,s, \ell)_{d}$-atom is just a $(\wp,q,s, \ell)$-atom.
\end{remark}
We can now state  the theorem of atomic decomposition in this context.

\begin{thm}\label{main1}
Let $1\leq\ell\leq n$ and $(\wp,q,s)$ be an admissible triple. Assume that  $\left\{\a_{j}\right\}^{\infty}_{j=1}$ is a sequence of multiples of $(\wp,q,s, \ell)_d$-atoms,  with each $\a_{j}$ related to the  ball $B_{j}$. Assume moreover  that
$$ \sum_{j}\wp\left(B_{j},\left\|\a_{j}\right\|_{L^{q}_{\wp}(B_{j},\Lambda^\ell)}\right )< \infty.$$
Then  the series $\sum_j \a_j$ converges in the distribution sense and in $\H_d^{\wp}(\bR^{n},\Lambda^{\ell})$. Moreover,
$$\left\|\sum_j \a_j\right\|_{{\H^{\wp}(\bR^{n},\Lambda^{\ell})}}\lesssim \mathfrak N_{q}(\left\{\a_{j}\right\})<\infty.$$
Conversely,  if $q<\infty$, and $f\in\H^{\wp}_d(\mathbb R^{n},\Lambda^{\ell})$,  there exists a sequence  $\left\{\a_{j}\right\}^{\infty}_{i=1}$ of multiples of  $(\wp,q,s, \ell)_{d}$-atoms, each $\a_{j}$ related to a ball $B_{j}$,  such that  
\begin{equation}
f=\sum_{j}\a_{j},\text{ in }\mathcal S'(\mathbb R^{n},\Lambda^{\ell}),\label{decompo2}
\end{equation}
with
 \begin{equation}
\sum_{j}\wp(B_{j},\left\|\a_j\right\|_{L^{q}_{\wp}(B_{j}, \Lambda^\ell)})<\infty.\label{size2}
\end{equation}
 Moreover, if $\left\|f\right\|_{\mathcal H^{\wp}(\bR^{n},\Lambda^{\ell})}=1$, then
\begin{equation}
 \mathfrak N_{q}(\left\{\a_{j}\right\})\sim 1.\label{equivalence2}
\end{equation}
\end{thm}

The first statement is a direct consequence of the corresponding statement when there is no condition on the exterior derivative. Indeed, this implies that  $\sum_j \a_j$ converges in the distribution sense and in $\H^{\wp}(\bR^{n},\Lambda^{\ell})$. Since all atoms are closed, the same is valid for the limit.

So only the converse, that is, the atomic decomposition, deserves a proof. We cannot use the proof given in \cite{Ky} since it would not guarantee that atoms are closed. As in \cite{AMR} we use atomic decomposition of tent spaces and follow the scheme given in \cite{CMS}. Tent spaces and their factorization have already been used in \cite{HYY} in the context of Hardy spaces of Musielak-Orlicz spaces to prove that $\H^{\wp}(\bR^n)$ can also be defined in terms of the area function.

Let us first recall the atomic decomposition in tent spaces.

 \subsection{Atomic decomposition in tent spaces}

 First, we recall the definitions related to tent spaces. For simplification we write definitions  in the scalar case, but they easily generalize to the case of forms.

 For all measurable functions $F$ on $\mathbb R^{n+1}_{+}$ and all $x\in\mathbb R^{n}$, let
\begin{equation*}
S(F)(x)=\left(\int_{\Gamma(x)}\left|F(y,t)\right|^{2}\frac{dydt}{t^{n+1}}\right)^{\frac{1}{2}},
\end{equation*}
with $\Gamma(x)=\left\{(y,t)\in\bR^{n+1}_{+}:\left|x-y\right|<t\right\}$. For $0<p<\infty$, Coifman, Meyer and  Stein \cite{CMS}, defined the tent space $\mathcal T^{p}(\mathbb R^{n+1}_{+})$ as the space of measurable functions $g$ satisfying  $\left\|g\right\|_{\mathcal T^{p}(\mathbb R^{n+1}_{+})}:=\left\|S(g)\right\|_{L^p(\R^n)}<\infty$. 
Here, we consider as in \cite{HYY} the following generalization of tent spaces.
\begin{defn}
Let $\wp$ be a growth function. The Musielak-Orlicz tent space  $\mathcal T^{\wp}(\mathbb R^{n+1}_{+})$ is the space of measurable functions $F$  on $\mathbb R^{n+1}_{+}$ such that $S(F)\in L^{\wp}(\R^n)$ with the (quasi-)norm defined by
\begin{equation*}
\left\|F\right\|_{\mathcal T^{\wp}(\mathbb R^{n+1}_{+})}\equiv\left\|S(F)\right\|_{L^{\wp}(\mathbb R^{n})}.
\end{equation*}
\end{defn}
We recall the notion of $(\wp,p)-$atom.
\begin{defn}
 Let $1<p<\infty$. A function $\A$ is called a $(\wp,p)-$atom related to the ball $B\subset \mathbb R^{n}$  if it satisfies the following condition.
\begin{itemize}
\item [(i)] $\A$ is supported in the tent over $B$, defined by
$$ \widehat{B}=\left\{(x,t)\in\bR^{n+1}_{+}:\left|x_{B}-x\right|+t\leq r_B\right\},$$
\item [(ii)] $\left\|\A\right\|_{\mathcal T^{p}(\mathbb R^{n+1}_{+})}\leq\left|B\right|^{1/p}\left\|\chi_{B}\right\|^{-1}_{L^{\wp}(\mathbb R^{n})}$.
\end{itemize}
Here $x_B$ is the center of $B$ and $r_B$ its radius.

We say that $\A$ is a $(\wp,\infty)-$ atom if it is a $(\wp,p)-$atom for all $1<p<\infty$.
\end{defn}
The following atomic decomposition result has been established in \cite{HYY} in the case of complex valued functions. It can be easily extended to vector valued functions, hence to forms, and we write it in this context.

\begin{thm}[Theorem 3.1 \cite{HYY}]\label{decomposition_form} 
   Let $\wp$ be a growth function. For any $F\in \mathcal T^{\wp}(\R^{n+1}_+, \Lambda^{\ell})$, there exist $\{\lambda_j\}_{j=1}^\infty\subset \mathbb C$ and a sequence $\left\{\A_{j}\right\}^{\infty}_{j=1}$  of $(\wp,\infty)$-atoms of  $\mathcal T^{\wp}(\R^{n+1}_+, \Lambda^{\ell})$, each $\A_j$ related to the ball $B_j$,   such that 
\begin{equation*}
F=\sum_{j} \lambda_j \A_{j}  \;\;\mbox{a. e. in $\R^{n+1}_+$}.
\end{equation*} 
 Moreover,
\begin{equation}\label{atoms-tent}
\widehat\N(\{\lambda_j \A_{j}\}):=\inf\left\{\lambda>0:\sum_j \wp\left(B_{j},\frac{\left|\lambda_{j}\right|}{\lambda\left\|\chi_{B_{j}}\right\|_{L^{\wp}(\mathbb R^{n})}}\right)\leq 1\right\}\lesssim\left\|F\right\|_{\mathcal T^{\wp}(\mathbb R^{n+1}_{+}, \Lambda^{\ell})}.
\end{equation}

\end{thm}
In all expressions, $|e|$ is now interpreted as the norm of $e$ when $e$ is a $\ell-$form.

As in the proof of \cite{CMS} we use the atomic decomposition for  functions $F$, which are defined by    
\begin{equation}
F(x, t)= f*\varphi_t(x).
\end{equation}

Let us now state two intermediate results of \cite{HYY}, which we will use with slightly different assumptions. 

\begin{lem}\label{controle2}
Let  $(\wp,q,s)$  be an admissible triple, and let  $\varphi\in\mathcal S(\mathbb R^{n})$ be supported in the unit ball centered at $0$.  Assume moreover that $\dst\int_{\R^n} x^{\gamma}\varphi(x)dx=0$ for $\left|\gamma\right|\le s$. 
Then, for all $f\in\H^{\wp}(\R^n)$,
\begin{equation*}
\Vert F\Vert _{\mathcal T^{\wp}(\R^{n+1}_+)}=\left\|S(F)\right\|_{L^{\wp}(\mathbb R^{n})}\lesssim\left\|f\right\|_{\H^{\wp}(\mathbb R^{n})},
\end{equation*}
where $F(x,t):= f*\varphi_t(x)$.
\end{lem}
\begin{proof}
One can find this result in \cite{HYY} under the additional assumptions that $\varphi$ is radial and  
 \begin{equation}\label{FourierHyp}
\int^{\infty}_{0}|\hat{\varphi}(t\xi)|^{2}\frac{dt}t=1
\end{equation} 
so that $S(F)$ coincides with the Lusin function of $f$, defined by
$$S_{\varphi}(f)(x):=\left(\int_{\Gamma(x)}\left|f\ast\varphi_{t}(y)\right|^{2}\frac{dydt}{t^{n+1}}\right)^{\frac{1}{2}}.$$
This allows them to prove that $f$ is in $\H^{\wp}(\R^n)$ if and only if  $S_{\varphi}(f)$ is in $L^{\wp}(\R^n)$. We are only interested by one implication and it 
 is clear that the   assumption \eqref{FourierHyp} can be relaxed into $$\sup_\xi\int^{\infty}_{0}|\hat{\varphi}(t\xi)|^{2}\frac{dt}t< \infty,$$     a positive lower bound being only necessary for the other implication. But this upper bound is a consequence of the moment conditions. The fact that the function is radial plays no role. One can follow the proof given in \cite{HYY} (see Equation (4.8)). Since the function $\varphi$ is only assumed to be in the Schwartz class and not necessarily compactly supported in \cite{HYY}, it is possible to have a simpler proof. Let us sketch it.   The main step  concerns the case when $f$ is a $(\wp, q, s)-$atom which is related to some ball $B$. Then $S_\varphi(f)$ is supported in  $\widetilde B$,  the ball with same center as $B$  and radius doubled. Moreover,  weighted $L^q$ estimates for the Lusin function imply that, uniformly in $t$, 
 $$\int_{\mathbb R^{n}}\left|S_\varphi(f)(x)\right|^{q}\wp(x,t)dx\lesssim \int_{\mathbb R^{n}}\left|f(x)\right|^{q}\wp(x,t)dx.$$
 Since $\wp(\widetilde B, t)\sim \wp( B, t)$, this implies that $\|S_\varphi(f)\|_{L_\wp^q(\widetilde B)}\lesssim \|f\|_{L_\wp^q( B)},$
 and, as a consequence, using \eqref{holder}, 
 $$\int_{\R^n} \wp(x, S_\varphi(f)(x))dx\lesssim \wp (B, \|f\|_{L_\wp^q( B)}).$$
 For general $f$ we use the atomic decomposition and \eqref{concave}.
 \end{proof}

 The next result may also be found in \cite{HYY}. 

 \begin{lem}\label{control}
 Let  $(\wp,q,s)$, $q<\infty$,  be an admissible triple and let  $\varphi\in\mathcal S(\mathbb R^{n})$ be supported in the unit ball centered at $0$.  Assume moreover that $\dst\int_{\R^n} x^{\gamma}\varphi(x)dx=0$ for $\left|\gamma\right|\le s$. 
Then, for any $(\wp,\infty)-$atom $\A$ in $\mathcal T^{\wp}(\R^{n+1}_+)$ related to the ball $B$,  the function 
\begin{equation}\label{Atoa}
\a(x):=\pi_{\varphi}(\A)(x):=\int_0^\infty (\A(\cdot,t)*\varphi_t)(x)\frac{dt}{t}
\end{equation}
is a $(\wp, q, s)$ atom of $\H^{\wp}(\R^n)$ related to the ball $\widetilde B$. Moreover, if $\left\{\A_{j}\right\}^{\infty}_{j=1}$ is a sequence of $(\wp,\infty)-$atoms of $\mathcal T^{\wp}(\R^{n+1}_+)$, each $\A_j$ related to the ball $B_j$ and if $\{\lambda_j\}_{j=1}^\infty$ is a sequence of scalars, then, for $\a_j=\pi_\varphi(\A_j)$, we have 
$$\N_q(\left\{\lambda_j\a_{j}\right\})\lesssim \widehat{\N}(\left\{\lambda_j\A_{j}\right\}).$$
\end{lem}

\begin{proof} We also give the proof for completeness. Let us consider an atom $\A$. It is elementary to verify that $\a$ satisfies the support and moment conditions. Let $p\in (q,\infty)$ We refer to \cite{CMS} for the inequality 
$$\|\a\|_p\lesssim \|\A\|_{\mathcal T^p(\mathbb R^{n+1}_{+})}.$$
As $A$ is a $(\wp,p)$-atom for every $p$, and satisfies $$\|\A\|_{\mathcal T^p(\mathbb R^{n+1}_{+})}\le |B|^{1/p}\left\|\chi_{B}\right\|^{-1}_{L^{\wp}(\mathbb R^{n})},$$
hence, to conclude for the size estimate of $\a$, it is sufficient to prove that, for all $t>0$,
$$\frac 1{\wp(B,t)}\int_{\widetilde B}| \a(x)|^q\wp(x,t)dx\leq C\|\a\|_p^q |B|^{-q/p}$$ for a uniform constant $C$. By  H\"older Inequality, written for the exponent $p/q$ and the conjugate exponent $r$, the left hand side is bounded by 
$$\frac {1}{\wp(B,t)}\left(\int_{\widetilde B}\wp(x, t)^r dx\right)^{1/r}\;\|\a\|_p^q.$$
We can conclude if we can prove that
$$\left(\frac 1{|\widetilde B|}\int_{\widetilde B}\wp(x, t)^r dx\right)^{1/r}\lesssim \frac{\wp(B,t)}{|B|}.$$
We recognize here H\"older's reverse Inequality, which is valid, uniformly in $t$, because of the assumption $\mathbb A_q$, for some $r>1$. So we fix $p=rq$.

The rest of the proof, that is, the passage to the atomic decomposition, is straightforward when using \eqref{concave}. 
 \end{proof}
 \begin{remark}
These two lemmas extend to vector valued distributions hence to $\mathcal T^{\wp}(\R^{n+1}_+, \Lambda^{\ell})$.
\end{remark}

 \subsection{Proof of the atomic decomposition}
 We now prove the converse part of Theorem \ref{main1}. The proof follows the lines of  \cite{LM1} but requires to be careful since we do not deal anymore with integrable functions. Let  $\varphi\in\mathcal C^{\infty}(\bR^{n})$ be a radial real-valued function, supported in the unit ball $\mathbb B$ centered at $0$, and  satisfying 
\begin{equation}\label{Fourier}
4\pi^2\int^{\infty}_{0}t\left|\xi\right|^{2}|\hat{\varphi}(t\xi)|^{2}dt=1 \qquad{\rm for \; all} \;\; \xi\neq 0\, , \xi\in \R^n
\end{equation}
and the moment conditions. 
 $$\int_{\R^n}\varphi(x) x^\gamma dx=0,\; 0\leq|\gamma|\leq s.$$ 
 This implies that, for   $k=1,\cdots, n$ and for any multi-index $\gamma$ with $|\gamma|\le s+1$, 
$$\int_{\R^n} \partial_k\varphi(x) x^\gamma dx=0.$$
It follows from computations of Fourier transforms that, for $f\in \mathcal S(\R^n)$, one has the identity
\begin{equation}\label{int-repr}
f=-\int_0^{\infty} f*\varphi_t*(\Delta \varphi)_t \frac{dt}{t}.
\end{equation}

Formula \eqref{int-repr} holds as well for $f$ a tempered distribution that vanishes weakly at $\infty$, see \cite{FS}. Namely, the integral from $1/N$ to $N$  in (\ref{int-repr}) tends, as a distribution, to $f$. Recall that a tempered distribution $f$ vanishes weakly at infinity if, for $\psi\in \mathcal S(\R^n)$, $f*\psi_t$ tends to $0$ at $\infty$ as a distribution. It is proved in \cite {HYY} that $f\in \H^\wp(\R^n)$ vanishes weakly at $\infty$, which proves formula  (\ref{int-repr}) for such $f$. It clearly extends to forms in $ \H^{\wp}_d(\R^n, \Lambda^\ell)$.
\smallskip

From now on, $f$ belongs to $ \H^{\wp}_d(\R^n, \Lambda^\ell)$. After integration by parts in \eqref{int-repr}, we apply ${\bf \Delta}$, the Hodge Laplacian defined in \eqref{Hodge}, to $f$. Here ${\bf \Delta}$, as well as $d$ and $\delta$ below, act only on the $x$ variable.
As $df=0$, 
$$f=-\int_0^{\infty} t^2 d (\delta f*\varphi_t)* \varphi_t \frac{dt}{t}.$$
Let us prove  that the function valued in $\Lambda^{\ell}$  given by $t\delta f*\varphi_t$ is in $\mathcal{T}^{\wp}(\R^{n+1}_+)$. Indeed, each coefficient can be written as a linear combination of functions $f*(\partial_k \varphi)_t$, and functions $\partial_k \varphi$  satisfy the assumptions of Lemma \ref{controle2}. So, by Theorem \ref{decomposition_form}, there exist $(\wp,\infty)$ atoms $\A_j$ and a sequence $(\lambda_j)$ with
$$t\delta f*\varphi_t=\sum_j \lambda_j \A_j(\cdot, t).$$
We finally define
$$\a_j(x)=-\int_0^\infty td \A_j(\cdot, t)*\varphi_t  \frac{dt}{t}.
$$
By Lemma \ref{control}, the required estimates of an atomic decomposition are satisfied. It remains to prove that the atoms $\a_j$'s are closed, and that  $\sum_j \lambda_j \a_j$ converges to $f$.

To prove that the atoms are closed, one has to prove that, for $\psi \in\mathcal S(\R^n, \Lambda^{\ell+1})$, $\langle \a_j, \delta \psi\rangle=0$. It is sufficient to prove that, for $\psi \in\mathcal S(\R^n, \Lambda^{\ell})$, 
\begin{equation}\label{commutation}
\langle \a_j, \psi\rangle = -\int_0^\infty t\langle d \A_j(\cdot, t)*\varphi_t, \psi\rangle  \frac{dt}{t},
\end{equation}   that is, integrations commute.
As functions $(\partial_k \varphi)_t$ are in  $\mathcal T^{p'}(\R^{n+1}_+)$, the function  
$$\int_0^\infty | td \A_j(\cdot, t)*\varphi_t|\frac{dt}{t}$$
is bounded because of the duality $\mathcal T^p -\mathcal T^{p'}$. So the double integral in \eqref{commutation} is absolutely convergent and we can exchange integrations as wished. 

Let  us finally prove the convergence in the distribution sense of  $\sum_j \lambda_j \a_j$. We adapt the proof given in \cite{HYY} in formula (4.24). Let $\psi$ be a compactly supported form in $\mathcal S(\R^n, \Lambda^{\ell})$ . Write

\begin{eqnarray*}
\langle f, \psi\rangle & = &-\langle \int_{0}^\infty  t^2  d (\delta f*\varphi_t)* \varphi_t \frac{dt}{t}, \psi\rangle
 = -\langle \lim_{N\rightarrow \infty}\int_{1/N}^N  t^2  d (\delta f*\varphi_t)* \varphi_t \frac{dt}{t}, \psi\rangle\\
&=&- \lim_{N\rightarrow \infty}\int_{1/N}^N  t^2 \langle d (\delta f*\varphi_t)* \varphi_t,  \psi\rangle\frac{dt}{t}
= -\int_{0}^\infty  t^2 \langle d (\delta f*\varphi_t), \varphi_t * \psi\rangle\frac{dt}{t},\\
&=&\lim _{N\rightarrow \infty}  \int_{0}^\infty  t \langle\sum_{j=1}^N\lambda_j \A_j, \varphi_t * \delta\psi\rangle\frac{dt}{t}=\lim _{N\rightarrow \infty} \langle\int_{0}^\infty  t \sum_{j=1}^N\lambda_j d\A_j* \varphi_t \frac{dt}{t}, \psi\rangle\\
&=&\lim _{N\rightarrow \infty} \langle \sum_{j=1}^N\lambda_j\a_j, \psi\rangle,
\end{eqnarray*}
where the first line is only the definition. The first equality of the second line follows from Fubini since $\psi$ is compactly supported and $t^2 \langle d (\delta f*\varphi_t)* \varphi_t$ is locally integrable. So the integral has a limit, which is the integral from $0$ to infinity. First equality in the third line comes from the assumption $\sum_j|\lambda_j|<\infty$ and from duality $\mathcal T^p -\mathcal T^{p'}$. Commutation between the integral and the duality bracket follows  as in \cite{HYY} from the absolute convergence of each term in the sum.

\smallskip

This allows us to conclude for the proof of atomic decomposition.

\begin{remark}
In order to decompose $f\in \H^\wp_d(\R^n, \Lambda^{\ell})$ it is also possible to copy the proof given for Hardy spaces of Musielak-Orlicz type of holomorphic functions \cite{BG}. One can start from the atomic decomposition of $f$ in $\H^\wp(\R^n, \Lambda^{\ell})$ and project each atom on closed forms, where the projection is the orthogonal projection in $L^2$, which may be written in terms of Riesz transforms. Then the projection of each atom is a molecule, so that we find a molecular decomposition of $f$ and not an atomic one.
\end{remark} 
\section{Duality}
Our aim here is to identify the dual space of $ \H^\wp_d(\R^n, \Lambda^{\ell})$ as a space of $BMO$ type. When one does not restrict to closed forms, the duality is obtained as in the scalar case and has been considered in \cite{Ky} when the growth function $\wp$ satisfies the inequality $nq(\wp)<(n+1)i(\wp)$, then, in all generality in \cite{LY}.   For simplicity we assume here that  $nq(\wp)<(n+1)i(\wp)$, so that the only moment condition on atoms is the requirement of zero mean.

With these conditions, the dual of  $\H^\wp(\R^n)$ identifies with $BMO^{\wp}(\mathbb R^{n})$, see \cite{Ky}. Let us write this duality when functions have values in $\Lambda^{\ell}$. For all $\ell$ we first define  $BMO^{\wp}(\mathbb R^{n},\Lambda^{\ell})$ as the space of locally integrable functions $g$ with values in $\Lambda^\ell$ such that
 \begin{equation*}
 \left\|g\right\|_{BMO^{\wp}}:=\sup_{B}\frac{1}{\left\|\chi_{B}\right\|_{L^{\wp}}}\int_{B}\left|g(x)-g_{B}\right|dx<\infty,
 \end{equation*}
 with $\dst g_{B}=\frac{1}{\left|B\right|}\int_{B}g(x)dx$.
 The duality may be written as follows. \\
 Recall that, because of the atomic decomposition, the subspace of functions that are compactly supported and bounded, valued in $\Lambda^{\ell}$, is dense $ \H^\wp(\R^n, \Lambda^{\ell})$.
 \begin{prop}
 Let  $\wp$ be a growth function with $nq(\wp)<(n+1)i(\wp)$. The dual space of $ \H^\wp(\R^n, \Lambda^{\ell})$ identifies with the space 
 $BMO^{\wp}(\mathbb R^{n},\Lambda^{n-\ell})$, with duality given by
 \begin{equation*}
L_g(f):= \langle g, f\rangle:=\int_{\R^n} f\wedge g
 \end{equation*}
 for $f$ compactly supported and bounded, valued in $\Lambda^{\ell}$. 
 \end{prop}
 We have chosen in this proposition to write the duality by using wedge products and integration of $n$ forms. We could have chosen to extend the duality bracket that we used up to now. If we had made this second choice, an element of the dual would belong to   $BMO^{\wp}(\mathbb R^{n},\Lambda^{\ell})$. Hodge duality allows to pass easily from one representation to the other.

Since $ \H^\wp_d(\R^n, \Lambda^{\ell})$ is a closed subspace of $ \H^\wp(\R^n, \Lambda^{\ell})$, its  dual   identifies with a quotient space mod its annihilator in   $BMO^{\wp}(\mathbb R^{n},\Lambda^{n-\ell})$. We first characterize this last one. 
\begin{lem}\label{dg=0} Let us assume that the growth function $\wp$ satisfies the condition $nq(\wp)<(n+1)i(\wp)$.
For $g\in BMO^\wp(\bR^n,\Lambda^{n-\ell})$,
the corresponding linear form $L_g$ annihilates $ \H^\wp_d(\R^n, \Lambda^{\ell})$ if and only if $dg=0$ in $\bR^n$.
\end{lem}
\begin{proof}
Assume that $\displaystyle\int_{\R^n} f\wedge g=0$ for all bounded functions $f$ with compact support. In particular $\displaystyle\int_{\R^n} d\varphi\wedge g=0$ for all $\varphi\in \mathcal S(\bR^n,\Lambda^{n-\ell-1})$ compactly supported. This implies that $dg=0$.

Conversely, assume that $g\in BMO^{\wp}(\mathbb R^{n},\Lambda^{\ell})$ is such that $dg=0$. Because of the atomic decomposition given in Theorem \ref{main1}, for any $f\in  \H^\wp_d(\R^n, \Lambda^{\ell})$, there exists a sequence $(\a_j)_{j=1}^\infty$ of  $(\wp,q,s, \ell)_d$-atoms, each $\a_{j}$ related to a ball $B_{j}$, with $f=\sum_{j} \lambda_j\a_j$. Moreover, because of the continuity of the linear form $L_g$,  
$$L_g(f)=\sum_j \lambda_j L_g(\a_j).$$ 
So it is sufficient to prove that each term vanishes, that is, $L_g$ annihilates all atoms $\a$. Let us recall some properties of functions of $BMO^{\wp}(\mathbb R^{n})$. 
For    $q>q(\wp)$, the following John-Nirenberg inequality holds (see  \cite{LY}).
\begin{equation}\label{JN} \sup_{B}\frac{1}{\left\|\chi_{B}\right\|_{L^{\wp}}}\int_{B}\left|g(x)-g_{B}\right|^{q'}\wp(x, \|\chi_B\|^{-1}_{L^{\wp}})^{1-q'}\; dx \lesssim \left\|g\right\|_{BMO^{\wp}}^{q'}. \end{equation}
As a consequence, we claim that:
\begin{lem}\label{Lg} Let $q>q(\wp)$.  Let $g$ in $BMO(\R^n)$ and $f$ of zero mean in $L^q_{\wp}(B)$, then 
the action of the linear form $L_g$ on $f$ is given by 
$$L_g(f)= \int_{\R^n} f(x)g(x) dx.$$
\end{lem}
This result extends easily to forms. 
We postpone its proof to end the proof of Lemma \ref{dg=0}. 
Let $\a$ be in $L^q(\R^n, \Lambda^{\ell})$ with compact support and $d\a=0$. As $g$ is locally in  $L^{q'}(\R^n, \Lambda^{n-\ell})$ and $dg=0$, one gets  directly  that 
$$L_g(\a)=\int_{\R^n}\a\wedge g =0,$$
which we wanted to prove.
Details are given in \cite{LM1}, where the particular case $\wp (x, t)=t$ is considered.
\end{proof}
It remains to prove Lemma \ref{Lg}.
\begin{proof}
Without loss of generality we can asssume that $f$ is a $(\wp, 0, q)$-atom. So the left hand side is well defined. The right hand side is well defined because of   \eqref{JN}. Indeed, we can replace $g$ by $g-g_B$ because of the moment condition. We then use H\"older's Inequality with $t= \|\chi_B\|^{-1}_{L^{\wp}}$, so that $\wp (B, t)=1$ (see \cite{Ky}), and write
\begin{eqnarray*} \int_{\R^n} |f(x)(g(x)-g_B)| dx  \hspace{10cm}\\  \qquad\qquad \leq \left(\int_{\R^n} |f(x)|^q\wp(x, t) dx \right)^{\frac 1q} \left( \int_{\R^n} |g(x)-g_B|^{q'}(\wp(x, t))^{1-q'} dx\right)^{1-\frac 1q},\end{eqnarray*} which is uniformly bounded by the definition of atoms for $f$ and \eqref{JN} for $g$.

As a consequence of this last inequality, and using the atomic decomposition in terms of $(\wp, q, 0,\ell)$-atoms,  the linear form $\displaystyle f\mapsto  \int_{\R^n} f(x)g(x) dx$ extends continuously from multiples of atoms to all $f\in \H^\wp(\R^n)$ and coincides with $L_g$ on bounded functions with zero mean and compact support. So it equals $L_g$, which ends the proof.
\end{proof}

Let us now identify the dual of $\H^{\wp}_{d}(\mathbb R^{n},\Lambda^{\ell})$ as a $BMO$ type space mod closed forms. Remark first that, thanks to the atomic decomposition, $L^q$ elements of $\H^{\wp}_{d}(\mathbb R^{n},\Lambda^{\ell})$ with compact support are dense for all $q(\wp)<q<\infty$.

\begin{defn} Let $(\wp, q, 0)$ be an admissible triple, $q'$ the conjugate exponent. For  $1\leq \ell<n$,  ${\widetilde{BM}\overset d O } \,^{\wp,q}(\mathbb R^{n},\Lambda^{\ell})$ is the space, mod closed forms, of locally $q'$-integrable functions $g$ valued in $\Lambda^\ell$ with
 \begin{equation*}
 \left\|g\right\|_{\widetilde{BMO}^{\wp, q}}:=\sup_{B}\inf_{\gamma\in L^{q'}_{\rm loc}(\R^n,\Lambda^\ell), d\gamma=0}\frac{1}{\left\|\chi_{B}\right\|_{L^{\wp}}}\int_{B}\left|g(x)-\gamma(x)\right|^{q'}\wp(x, \|\chi_B\|^{-1}_{L^{\wp}})^{1-q'}\;dx
 \end{equation*}
  finite. \end{defn} 
 The semi-norm of $g$ is equal to $0$ if and only if $g$ coincides locally with a closed form, that is, $g$ is a closed form. It follows from the following proposition that this space does not depend on $q>q(\wp)$. 
   \begin{prop} Let $(\wp, q, 0)$ be an admissible triple, $q'$ the conjugate exponent. Then
 ${\widetilde{BM}\overset d O } \,^{\wp,q}(\mathbb R^{n},\Lambda^{n-\ell})$ identifies with the dual of $\H^{\wp}_{d}(\mathbb R^{n},\Lambda^{\ell})$. Namely, if  $f$ has the atomic decomposition  $f=\sum_j\lambda_j \a_j$ with $(\a_j)_{j=1}^\infty $ a sequence of $(\wp, q, 0, \ell)_d-$atoms and if $g$ is in ${\widetilde{BM}\overset d O } \,^{\wp,q}(\mathbb R^{n},\Lambda^{n-\ell})$,  then 
\begin{equation}
L_g(f)=\sum_j \lambda_j\int_{\mathbb R^{n}} \a_j \wedge g.\label{dualexact}
\end{equation}
\end{prop}
\begin{remark}This proposition generalizes Theorem 3.3 in \cite{LM1}.
\end{remark}

\begin{proof}
We do not give a detailed proof since the result follows from slight modifications of the arguments given for Lemma \ref{dg=0}. It is clear that $L_g$ depends only on the equivalent class of $g$ because of the previous lemma. The fact that the sum in \eqref{dualexact} is finite follows from the same arguments as the one given before. Conversely, by Hahn Banach Theorem, a continuous linear form is given by some function $g$ in $BMO^{\wp}(\mathbb R^{n},\Lambda^{n-\ell})$. The equivalence class of $g$ mod closed forms is in ${\widetilde{BM}\overset d O } \,^{\wp,q}(\mathbb R^{n},\Lambda^{n-\ell})$.
\end{proof} 

\begin{remark} The dual space of $\H^{\wp}_{d}(\mathbb R^{n},\Lambda^{\ell})$ identifies also with the space  $BMO^\wp_\delta(\bR^n,\Lambda^{n-\ell})$, where the index $\delta$ stands for the restriction to $\delta$ closed forms. Indeed,  in the equivalence class of a function $g\in BMO^\wp(\bR^n,\Lambda^{n-\ell})$, there is one and only one $\delta$ closed-form: the uniqueness comes from the fact that only constants are such that  $\delta g=dg=0$. On the other hand, any function $g$ may be written as the sum of a $\delta$-closed form and a $d$-closed form: write $f=\delta d{\bf \Delta}^{-1}+d\delta {\bf \Delta}^{-1}f$. Here $\bf \Delta$ is the Hodge Laplacian. The second one is closed. Each of them is in $BMO^\wp(\bR^n,\Lambda^{n-\ell})$ since their coefficients are obtained through products of Riesz transforms.
\end{remark}

The continuity of Riesz transforms on $BMO^\wp(\bR^n)$ is a consequence of their continuity in  $\H^{\wp}(\R^n)$, see \cite{CCYY}. Remark that this implies that there is a bounded projection of $\H^{\wp}(\R^n, \Lambda^{\ell})$ onto the subspace that consists  of closed forms.

\section{Generalized div-curl Lemma }

For the particular growth function $\theta$ defined in (\ref{orlicz}) by
$$\theta(x,t)=\frac{t}{\log(e+|x|)+\log(e+t)},$$ we can state a generalized div-curl lemma (Theorem \ref{main2bis}) and a weak factorization decomposition (see next section), which may be seen as  a weak converse.
\begin{thm}\label{main2bis}{\rm \bf[Generalized div-curl Lemma]}
Let $1\leq \ell, m \leq n-1$. Let $u\in \H^1_d(\bR^{n},\Lambda^{\ell})$ and $v\in BMO_d(\bR^{n},\Lambda^{m})$. Then $u\wedge v$ belongs to $\H^{\log}_{d}(\bR^{n},\Lambda^{\ell+m})$.
\end{thm}
\begin{proof}[Proof of Theorem \ref{main2bis}] The wedge product is taken in the distribution sense, as defined in \cite{BIJZ}. By eventually taking wedge products with forms $dx_I$, it is sufficient to consider the case when $\ell+m=n$.  Since the Hodge Laplacian ${\bf \Delta}$ commutes with $d$, we can write $u={\bf \Delta} ^{-1}d\delta u=d{\bf \Delta}^{-1/2}w$ with $w={\bf \Delta}^{-1/2}\delta u$.  As the operator $ {\bf \Delta}^{-1/2}\delta $ may be easily written in terms of Riesz transforms, $w$ is a $\ell-1$ form with coefficients in $\H^1(\bR^n)$. Hence
$$u=\sum_{j, I} R_j( w_I) e^j \wedge e^I$$
with $w_I\in   \H^1(\bR^n)$.    
The rest of the proof is an easy adaptation of the particular case of the div-curl lemma proved in \cite{BGK}. We first recall the notations and main result of \cite{BGK}. There exist two linear operators $S$ and $T$ defined in terms of a wavelet basis so that
\begin{eqnarray*}
S&:&\, BMO(\mathbb R^{n})\times\H^{1}(\mathbb R^{n}) \rightarrow L^{1}(\mathbb R^{n}),\\
T&:&\, BMO(\mathbb R^{n})\times\H^{1}(\mathbb R^{n})\rightarrow \H^{\log}(\mathbb R^{n})
\end{eqnarray*} and 
$$bh=S(b,h)+T(b,h).$$ Now, one can write 
\begin{eqnarray*}
u\wedge v&=&d{\bf \Delta}^{-1/2} w\wedge v=\sum_{j\cup I\cup J=\{1,\dots,n\}} R_j (w_I)  v_Je^j\wedge e^I\wedge e^J\\
&=&\sum_{j\cup I\cup J=\{1,\dots,n\}}(S( R_j (w_I), v_J)+T( R_j( w_I),v_J))e^j\wedge e^I \wedge e^J\\
&=&\sum_{j\cup I\cup J=\{1,\dots,n\}}(S( R_j( w_I), v_J)+S(w_I,R_j(v_J))+T( R_j (w_I),v_J))e^j\wedge e^I\wedge e^J.
\end{eqnarray*}
The last equality holds since, as $dV=0$,
$$0=\sum_{j, J}S(w_I,R_j(v_J)e^j)\wedge e^J.$$
It is proved in \cite{BGK} that  $S( R_j (w_I), v_J)+S(w_I,R_j(v_J))$ and $T( R_j (w_I),v_J)$ belong to $\H^{\log}(\bR^n)$. This allows to conclude.\end{proof}
\medskip
\section{Weak Factorization} 
Let us now consider weak factorization.
\begin{thm}\label{main2}{\rm \bf [Weak factorization]}
Let $1\leq \ell, m \leq n-1$. Any  $f\in\H^{log}(\bR^{n},\Lambda^{\ell+m})$ such that $df=0$ can be written as
\begin{equation*}
f=\sum^{\infty}_{k=0}u_{k}\wedge v_{k}\text{ in the sense of distributions}
\end{equation*}
with

$u_{k}\in \H^1(\bR^{n},\Lambda^{\ell})$ and $v_{k}\in BMO(\bR^{n},\Lambda^{m})$, $du_k=0$ and $dv_k=0$.
Moreover, 
\begin{equation}
\sum_k \|u_{k}\|_{\H^1(\bR^{n},\Lambda^{\ell})}\| v_{k}\|_{BMO^+(\bR^{n},\Lambda^{m})}\leq C\|f\|_{\H^{\log}(\bR^{n},\Lambda^{\ell+m})}.
\end{equation}
\end{thm}

\medskip

 \begin{proof} The proof follows the same scheme as in \cite{LM1} but requires 
first to build adapted functions  of bounded mean oscillation. We consider separately small balls or balls that are far from the origin on one hand, large balls that contain the origin or are close to contain it, on the other hand.
 
 Recall that, for $B$ a ball in $\R^n$, we denote by $x_B$ its center  and  $r_B$ its radius.
 \smallskip
 
\subsection{Proof in the case I: $r_B\leq \min(1, \frac{|x_B|}2)$.} 
 
 \begin{lem}\label{G_BMO} Let $B$ be a ball  in $\R^n$ with center $x_B$ and radius $r_B$ with $r_B\leq \min(1, \frac{|x_B|}2)$. There exists a function $G$ in $BMO$ with the following properties 
 	\begin{enumerate}
 	\item $G$ is  constant on $B$,
 	\item $G(x)\geq C^{-1}(\log(e+|x_B|)+|\log(r_B)|)$ for all $x\in B,$
 	\item $\|G\|_{BMO^+} \leq C$
 	\end{enumerate} 
	for some constant independent of $B$.
\end{lem}
 		
 \begin{proof}  If $\frac 1{r_B}\le \frac {|x_B|}2$, it is sufficient to find a function bounded from below by $c\log(e+|x_B|)$.  We take
 \begin{equation}\label{deux}
 G(x)= \min (\log (e+2|x|), \log(e+|x_B|)).
 \end{equation}
 which is clearly constant and equal to $\log(e+|x_B|)$ on $B$. It belongs to $BMO$ as the minimum of two $BMO$ functions (see Garnett's book \cite{G}). Using the bound $\log (e+2|x|)$ allows us to get that the integral of $G$ on the unit cube is bounded. Hence,  $\|G\|_{BMO^+} \leq C$.
 
 If $\frac{|x_B|}2<\frac{1}{r_B}$, we take
  \begin{equation}\label{trois}
 G(x)= \min (\log (e+|x-x_B|^{-1}), \log(e+r_B^{-1})).
 \end{equation}
Analogous arguments allow us to show that $G$ satisfies the required properties.
  \end{proof}
 Let us now consider  the factorization of atoms in this particular case.
 
 \begin{lem}\label{lem-fact}
 Let $1\leq\ell, m\leq n-1$ and $\a$ be a $(\wp, q, 0)-$atom of $\H_d^{\log}(\R^n, \Lambda^{\ell+m})$ related to a ball $B$ satisfying $r_B\leq \min(1, \frac{|x_B|}2)$.  
 
 Then, $\a$ may be written as a  sum of $\frac{n!}{(\ell+m-1)!(n-\ell-m+1)!}$ terms 
  $u_j\wedge v_j$, with $u_j$ a $q$-atom of $\H^1_d(\bR^{n},\Lambda^{\ell})$ and $\|v_j\|_{BMO^+(\bR^{n})}$ uniformly bounded.
 \end{lem}
 
 \begin{proof} We first give the proof for $\ell=m=1$ and prove that it implies the general case. The $2$-form atom $\a$ is assumed to be supported in a ball $B$, of center $x_B$ and radius $r_B$. As the whole problem is invariant by rotation, we can, without loss of  generality, assume that all coordinates of $x_B$ have same modulus, equal to $(\sqrt n)^{-1}|x_B|$. If we apply Lemma \ref{G_BMO} in one variable, for the projected ball, we find that, for each $k$, there exists a function $G_k$, which depends only on $x_k$, and satisfies the same properties as $G$. We call $\gamma_k$ its (constant) value on $B$.
 
 The atom $\a$ can be written as the exterior derivative of a $1$-form. Indeed, for $n>2$ we use the fact that $d$ is closed, while for $n=2$ we use the fact that it is of mean $0$. More precisely, we can write  $\a=d\mathfrak b$, where $\mathfrak b$ is also supported in $B$ and has the  $L^{ q}$ norm of its gradient  bounded (up to a uniform constant $C$) by the $L^q$ norm of $\a$ (see \cite{Sc}). Hence  $$\a=\sum_j\sum_k \frac{\partial \varphi_k}{\partial x_j} dx_j\wedge dx_k:=\sum_k{a_k},$$
 with $a_k:=\sum_j \frac{\partial \varphi_k}{\partial x_j} dx_j\wedge dx_k$. Moreover,
 $$\|a_k\|_{L^q(B)}\le \frac {\log(\frac{1}{r_B} + r_B)+\log(e+|x_B|)}{|B |^{1-1/q}}\leq \frac{C\gamma_k}{|B|^{1-1/q}}.$$
 For each fixed $k$, we will write $a_k$ as a product, as required.  We take
  $$u_k=\gamma_k^{-1}\sum_j \frac{\partial \varphi_k}{\partial x_j} dx_j, $$
 which is clearly a closed form as a differential. It is (up to a uniform constant) an atom of $\H_d^{1}(\R^n, \Lambda^{1}).$
 We then take $v_k=G_k dx_k$. It is a closed form since $G_k$ depends only on the variable $x_k$.  On the ball $B$, which contains the support of $u_k$, it is equal to $\gamma_k dx_k$, so that $$u_k\wedge v_k=\sum_j \frac{\partial \varphi_k}{\partial x_j} dx_j\wedge dx_k.$$
 This concludes the proof for $\ell=m=1$.
 
 In the general case, we write 
 $$\a=\sum_j\sum_I \frac{\partial \varphi_I}{\partial x_j} dx_j\wedge dx_I:=\sum_I a_I$$
 with $I$ of length $\ell+m-1$  and $\varphi_I$ supported in $B$. Moreover we can assume that the coefficients $a_I$'s satisfy the same inequality as those of $\a$. We decompose each term into a wedge product,
  the number of terms being equal to the number of possible choices for $I$.  If we choose the index $k$ as the larger index in $I$  and write  $I=I'\cup I'' \cup \{k\}$ with $|I'|=\ell-1$ and $|I''|=m-1$, the two-form  
$$ \sum_j \frac{\partial \varphi_I}{\partial x_j} dx_j\wedge dx_k$$
is a closed form, which we can write as a wedge product $u\wedge v$.  Then $a_I$ itself is (up to the sign) the wedge product $(u\wedge dx_{I'})\wedge (v\wedge dx_{I''}).$ It is clear that each term is closed. Estimates are straightforward.
\end{proof}

\subsection{Proof in the case II: $r_B\geq \min(1, \frac{|x_B|}2)$.}

The analogous of Lemma \ref{lem-fact} is the following.
 \begin{lem}\label{lem-fact2}
 Let $1\leq\ell, m\leq n-1$ and $\a$ be a $(\wp, q, 0)-$atom of $\H_d^{\log}(\R^n, \Lambda^{\ell+m})$ related to a ball $B$ with $r_B\geq \min(1, \frac{|x_B|}2)$. Then, for $1<r<q$, the atom $\a$ may be written as a  sum of $\frac{n!}{(\ell+m-1)!(n-\ell-m+1)!}$ terms 
  $u_j\wedge v_j$, with $u_j$ anŽŽ $r-$atom of ${\H^1_d(\bR^{n},\Lambda^{\ell})}$ and $\|v_j\|_{BMO^+(\bR^{n})}$ uniformly bounded.
 \end{lem}
 \begin{proof}
 We have seen that it is sufficient to do it for $2-$forms and take the same notations as in the proof of the previous lemma. So we want again to factorize $a_k:=\sum_j \frac{\partial \varphi_k}{\partial x_j} dx_j\wedge dx_k=d\varphi_k\wedge dx_k.$ By assumption, we have the inequality
$$\|\ d \varphi_k\|_{q}\le C\frac {\log(e + r_B)}{|B |^{1-\frac 1q}}.$$
Recall that the equation $d\psi=d\varphi_k$ has a solution in the Sobolev space $W^{1,q}(\R^n, \Lambda^{\ell-1})$  supported in $B$, see again \cite{Sc}. Moreover it satisfies the inequalities
\begin{eqnarray*}
\|\ d \psi\|_{q} &\le &C\frac {\log(e + r_B)}{|B |^{1-\frac 1q}},\\
\|\ \psi\|_{q}&\le&  Cr_B\frac {\log(e + r_B)}{|B |^{1-\frac 1q}}.
\end{eqnarray*}
 We take $G(x):=\log (e+x_k^2)$, which has a bounded norm in $BMO^+$. If we define $u:=d(G^{-1}\psi)$ and $v:=Gdx_k$, it follows from elementary computations that $u\wedge v= a_k$. Both are closed forms. It remains to prove that $u$ satisfies the required estimates (up to a uniform constant) for an $r-$atom of $\H^1$. We develop the exterior derivative and have to consider that the coefficients of $G^{-1}d\psi $ on one hand, $G^{-2}dG\wedge \psi$ on the other hand, are in $L^r(\R^n)$. This is a consequence of H\"older's Inequality, using the elementary inequalities, valid for all $s>1$,
 $$\int_B G(x)^{-s} dx\leq C \frac{|B |} {(\log(e + r_B))^s}\qquad \qquad \int_B G(x)^{-2s} \frac{dx}{1+|x|} \leq C \frac{|B |} {(\log(e + r_B))^{2s}}.$$
\end{proof}
It remains to end the proof of Theorem \ref{main2}.
Let us consider $f\in \H_d^{\log}(\R^n, \Lambda^{\ell+m}),$ with  $1\leq\ell, m\leq n-1.$ Without loss of generality, we assume that its norm is $1$. We write $f=\sum_{\tau}\lambda_\tau \a_{\tau}$, where  $\a_{\tau}=$ are $(\wp, q, s, \ell)_d$-atoms of $\H_d^{\log}(\R^n, \Lambda^{\ell+m})$ related to the balls $B_\tau$. Moreover, 
$$ \sum_{\tau}\theta\left(B_{\tau},\left\|\lambda_\tau\a_{\tau}\right\|_{L^{q}_{\wp}(B_{\tau}, \Lambda^{\ell+m})}\right )\lesssim 1.$$ 
It follows from the properties of growth functions that  
\begin{equation*}
\sum_\tau |\lambda_\tau|\lesssim 1. 
\end{equation*} Each atom is written as a finite sum of products, the number of which depending only of $n, \ell, m$, with the product of norms uniformly bounded. This allows to conclude for the inequality.   
 \end{proof}
 
\begin{remark} Theorem \ref{main2} is not completely satisfactory since the condition 
\begin{equation*}
\sum_k \|u_{k}\|_{\H^1(\R^n, \Lambda^{\ell})}\| v_{k}(\R^n, \Lambda^{m})\|_{BMO^+}\leq C
\end{equation*}
does not imply that the sum $\sum_k u_k\wedge v_k $ is in $\H^{\log}(\R^n, \Lambda^{\ell+m})$ as in the similar factorization of $\H^1$. It only implies that it belongs to the smallest Banach space containing $\H^{\log}(\R^n, \Lambda^{\ell+m})$. \end{remark}

We finish this section by the particular case of scalar products of vector fields. 
\begin{cor}
Let $n\geq 2$ and    $f\in\H^{\log}(\bR^{n})$. Then $f$  can be written as
\begin{equation*}
f=\sum^{\infty}_{k=0}F_{k}\cdot G_{k}\text{ in the sense of distributions}
\end{equation*}
with
$F_{k}\in \H^1(\bR^{n},\R^n)$ and $G_{k}\in BMO(\bR^{n},\R^{n})$ two vector fields  such that one of them is $\div-$free, the other one is $\curl-$free.
Moreover, 
\begin{equation}
\sum_k \|F_{k}\|_{\H^1(\bR^{n},\R^{n})}\| G_{k}\|_{BMO^+(\bR^{n},\R^{n})}\leq C\|f\|_{\H^{\log}(\bR^{n})}.
\end{equation}
\end{cor}

{\bf Acknowledgements.} The paper was completed when the fourth author was visiting
to Vietnam Institute for Advanced Study in Mathematics (VIASM). He would like
to thank the VIASM for financial support and hospitality.

\end{document}